\documentclass[11pt]{amsart}
\usepackage{amscd,amsfonts,mathrsfs,amssymb,amsmath}
\usepackage{hyperref}
\usepackage[margin=3.2 cm]{geometry}

\def\H{\mathbb H}
\def \h{{\bf H}_{\mathbb H}}

\def\C{\mathbb C}

\def\P{\mathbb P}

\def\Sp{\mathrm{Sp}}
\def\SU{\mathrm{SU}}

\def\SL{{\rm SL}}

\def\z{{\bf z}}

\newtheorem{theorem}{Theorem}[section]
\newtheorem{lemma}[theorem]{Lemma}
\theoremstyle{definition}

\theoremstyle{remark}

\numberwithin{equation}{section}
\theoremstyle{plain}

\newtheorem{cor}[theorem]{Corollary}

\newcommand{\secref}[1]{section~\ref{#1}}
\newcommand{\thmref}[1]{Theorem~\ref{#1}}

\newcommand{\corref}[1]{Corollary~\ref{#1}}
\newcommand{\eqnref}[1]{~{\textrm(\ref{#1})}}

\begin{document}

\title[On discreteness of subgroups of quaternionic hyperbolic isometries]{On discreteness of subgroups of quaternionic hyperbolic isometries} 
\author[K.  Gongopadhyay]{Krishnendu Gongopadhyay}
 \address{ Indian Institute of Science Education and Research (IISER) Mohali,
Knowledge City, Sector 81, SAS Nagar, Punjab 140306, India}
\email{krishnendu@iisermohali.ac.in, krishnendug@gmail.com}
\author[M. M. Mishra]{Mukund Madhav Mishra}
\address{Department of Mathematics, Hansraj College, University of Delhi, Delhi 110007, India}
\email{mukund.math@gmail.com}
\author[D. Tiwari]{Devendra Tiwari}
\address{Department of Mathematics, University of Delhi, Delhi 110007, India}
\email{devendra9.dev@gmail.com}
\subjclass[2000]{Primary 20H10; Secondary 51M10}
\keywords{hyperbolic space, J\o{}rgensen inequality, discreteness, quaternions.  }
\date{\today}
\thanks{Gongopadhyay acknowledges partial support from SERB MATRICS grant MTR/2017/000355.  \\Tiwari is supported by NBHM-SRF}
\begin{abstract}
Let $\h^n$ denote the $n$-dimensional quaternionic hyperbolic space. The linear group $\Sp(n,1)$ acts by the isometries of $\h^n$.  A subgroup $G$ of $\Sp(n,1)$ is called \emph{Zariski dense} if it does not fix a point on $\h^n \cup \partial \h^n$ and neither it preserves a totally geodesic subspace of $\h^n$. 
We prove that a Zariski dense subgroup $G$ of $\Sp(n,1)$ is discrete if for every loxodromic element $g \in G$ the two generator subgroup $\langle f, g f g^{-1} \rangle$  is discrete, where the generator $f \in \Sp(n,1)$ is certain fixed element not necessarily from $G$.

\end{abstract}
\maketitle
\section{Introduction}
The classical J\o{}rgensen inequality, see \cite{j},  gives a necessary criterion to check discreteness of a two generator subgroup of $\SL(2, \C)$ that acts by M\"obius transformations on the Riemann sphere. This has been generalized to the higher dimensional M\"obius group that acts on the $n$ dimensional real hyperbolic space, by several authors. A well-known consequence of the generalized J\o{}rgensen inequalities says that a subgroup $G$ of the M\"obius group is discrete if and only if every two generator subgroup is discrete, e.g. \cite{martin}, \cite{ah}. There have been several refinements of this result to obtain several discreteness criteria in M\"obius groups, e.g. \cite{mchen}, \cite{wlc}, \cite{gm2}. Several generalizations of the J\o{}rgensen inequality and related discreteness criteria have been obtained in further generalized setting like the complex hyperbolic space and  normed spaces, e.g. \cite{jkp}, \cite{fh}, \cite{mp2}, \cite{mp}.

Let $\H$ denote the division ring of Hamilton's quaternions. Let $\h^n$ denote the $n$-dimensional quaternionic hyperbolic space. Let $\Sp(n,1)$ be the linear group that acts on $\h^n$ by isometries.  In this paper following the above theme of research, we give discreteness criteria for a subgroup of  $\Sp(n,1)$. The arguments restrict over the complex numbers as well, and as a corollary we obtain discreteness criteria in $\SU(n,1)$. To state our main result we need the following notions. 

Recall that an element $g \in \Sp(n,1)$ is \emph{elliptic} if it has a fixed point on $\h^n$. It is \emph{parabolic}, resp. \emph{loxodromic} (or hyperbolic), if it has a unique fixed point, resp. exactly two fixed points on the boundary $\partial \h^n$. A unipotent parabolic element,  that is, a parabolic element  having all eigenvalues $1$,  is called a \emph{Heisenberg translation}.  It is well-known, see \cite{chen}, that an elliptic or loxodromic isometry $g$ is conjugate to a diagonal element in $\Sp(n,1)$.  If $g$ is elliptic, then up to conjugacy, 
\begin{equation} \label{elc} g=\mathrm{diag} (\lambda_1, \ldots, \lambda_{n+1}),\end{equation} 
where for each $i$, $|\lambda_i|=1$, and the eigenvalue $\lambda_1$ is such that the corresponding eigenvector has negative Hermitian length, while all other eigenvectors have positive Hermitian length. 
An elliptic element $g$ is called \emph{regular} if it has mutually disjoint classes of eigenvalues.  A regular elliptic element has a unique fixed point on $\h^n$. 
 If $g$ is loxodromic, then we may assume up to conjugacy, 
\begin{equation}\label{elp}
g=\mathrm{diag}(\lambda_1, \bar \lambda_1^{-1}, \lambda_3, \ldots, \lambda_{n+1}).
\end{equation}
with $|\lambda_1|>1$. 
One may associate certain conjugacy invariants to isometries as follows. 

For $g$ is elliptic, we define
\begin{equation} \label{eqq1} \delta(g)=\max\{ \ |\lambda_1-1|+|\lambda_i-1| \ : \ i=2, \ldots, n+1\}. \end{equation} 

For $g$ loxodromic, following \cite{cp1}, define the following quantities: 
$$\delta_{cp}(g)=\max\{|\lambda_i-1| \ : \ i=3, \ldots, n+1\}, \hbox{ and} $$
$$M_g=2 \delta_{cp}(g)+|\lambda_1-1|+|\bar \lambda_1^{-1}-1|.$$

Let $T_{s, \zeta}$ be a Heisenberg translation in $\Sp(n,1)$. We may assume up to conjugacy that (cf. \cite[p. 70]{chen})
\begin{equation} \label{he} 
T_{s, \zeta}=\begin{pmatrix} 1 & 0 & 0 \\ s & 1 & {\zeta}^{\ast} \\ \zeta & 0 & I \end{pmatrix}, 
\end{equation}
where $Re(s)=\frac{1}{2} |\zeta|^2$. 

\medskip  A subgroup  $G$ of $\Sp(n,1)$ is called \emph{Zariski dense} if it does not fix a point on $\h^n \cup \partial \h^n$, and neither it preserves a totally geodesic subspace of $\h^n$.  With the above notations, we prove the following theorem. 
 \begin{theorem}\label{thd} 
Let $G$ be a Zariski dense subgroup of  $\Sp(n, 1)$. 
\begin{enumerate}
\item 
Let $g \in \Sp(n, 1)$ be a regular elliptic element such that $\delta(g)<1$.  If $\langle g, hgh^{-1} \rangle$ is discrete for every loxodromic element $h \in G$, then $G$ is discrete.

\item Let $ g \in \Sp(n, 1)$ be loxodromic element such that $M_g < 1$. If $\langle g, hgh^{-1} \rangle$ is discrete for every loxodromic element $h \in G$, then $G$ is discrete.

\item Let  $g \in \Sp(n, 1)$ be a  Heisenberg translation such that $|\zeta| < \frac{1}{2}$. If $\langle g,   hgh^{-1} \rangle$ is discrete for every loxodromic $h$ in $G$, then $G$ is discrete. 
\end{enumerate} 
\end{theorem} 
Restricting everything over the complex numbers,  the above theorem also holds for $\SU(n,1)$. 
We also have the following.   
\begin{cor}\label{cor1}
Let $G$ be Zariski dense in $\Sp(n, 1)$, resp. $\SU(n,1)$. 
\begin{enumerate}
\item 
Let $g \in \Sp(n, 1)$, resp. $\SU(n,1)$,  be a regular elliptic element such that $\delta(g)<1$.  If $\langle g, hgh^{-1}  \rangle$ is discrete for every regular elliptic  element $h \in G$, then $G$ is discrete.

\item Let $ g \in \Sp(n, 1)$, resp. $\SU(n,1)$,  be loxodromic element such that $M_g < 1$. If $\langle g, hgh^{-1}  \rangle$ is discrete for every regular elliptic $h \in G$,  then $G$ is discrete.

\item Let  $g \in \Sp(n, 1)$, resp. $\SU(n,1)$,  be a  Heisenberg translation such that $|\zeta| < \frac{1}{2}$. If $\langle g, hgh^{-1}  \rangle$ is discrete for every regular elliptic $h$ in $G$, then $G$ is discrete. 
\end{enumerate} 
\end{cor} 

Note that the above results show that the discreteness of a Zariski dense subgroup $G$ of $\Sp(n,1)$, or $\SU(n,1)$, is determined by the two generator subgroups 
$\langle g, hgh^{-1} \rangle$, where  $h \in G$, but the generator $g$ is fixed and need not be an element from $G$, and also it is enough to take $h$ to be loxodromic or regular elliptic. After fixing such a `test map' $g$, conjugates of $g$ by generic elements of $G$  determine the discreteness. For  isometries of the real hyperbolic space, similar discreteness criteria using a test map and conjugates of it, have been obtained in \cite{yang2}, \cite[Theorem 1.2]{gm2} and \cite{gm}.  \thmref{thd} and \corref{cor1} generalize these works in $\Sp(n,1)$ and $\SU(n,1)$ respectively. 

\medskip We have noted down some preliminary notions in \secref{prel}. The main result has been proved in \secref{pfthd}.  To prove the results, we shall use some  generalized  J\o{}rgensen inequalities in $\Sp(n,1)$. We shall  use the J\o{}rgensen inequality of Cao and Parker \cite{cp1} for dealing with the subgroups having a loxodromic generator. For the subgroups having a unipotent parabolic generator, we shall follow a quaternionic version of  the Shimizu's lemma following Hersonsky and Paulin \cite{hp}. To deal with the subgroups having a regular elliptic generator, we shall use a variation of the inequality of Cao and Tan in \cite{ct1}.  In this case, we have introduced a new conjugacy invariant  $\delta(g)$ as given above. This invariant $\delta(g)$ is different from the conjugacy invariant $\delta_{ct}(g)$ used by Cao and Tan. The invariant $\delta(g)$ may be considered as a restriction of the Cao-Parker invariant $\delta_{cp}(g)$ to the subgroups having at least one generator elliptic. This new invariant also gives quantitatively better bound in a larger domain.   We refer to \secref{ctjo} for more details. 

\section{Preliminaries} \label{prel}
\subsection{The quaternionic hyperbolic space} 
We begin with some background
material on quaternionic hyperbolic geometry. Much of this can be
found in \cite{chen,kp}.

Let $\H^{n,1}$ be the right vector space over $\H$ of quaternionic dimension $(n+1)$  (so real dimension $4n+4$) equipped with the quaternionic Hermitian form for $z=(z_0,...,z_n), \ w=(w_0,...,w_n)$, 
$$\langle z, w \rangle=-( \bar z_0 w_1 +  \bar z_1 w_0) + \Sigma_{i=2}^{n} \bar z_i w_i. $$
Thus the Hermitian form is defind by the matrix 
$$J_2=\begin{pmatrix} 0 & -1 & 0 \\ -1 & 0 & 0 \\ 0 & 0 & I_{n-2} \end{pmatrix}.$$ 

Equivalently, one may also use the Hermitian form given by the following matrix wherever convenient. 
$$J_1=\begin{pmatrix} -1 &0 \\ 0 & I_{n} \end{pmatrix}.$$ 

Following Section 2 of \cite{chen}, let
\begin{eqnarray*}
V_0 =  \Bigl\{{\bf z} \in  \H^{n,1}-\{0\}:
\langle{\bf z},\,{\bf z}\rangle=0\Bigr\},\,\,
V_{-} = \Bigl\{{\bf z} \in \H^{n,1}:\langle{\bf z},\,{\bf
z}\rangle<0\Bigr\}.
\end{eqnarray*}
It is obvious that $V_0$ and $V_{-}$ are invariant under ${\rm Sp}(n,1)$.
We define an equivalence relation $\sim$ on $\H^{n,1}$ by
${\bf z}\sim{\bf w}$ if and only if there exists a non-zero
quaternion $\lambda$ so that ${\bf w}={\bf z}\lambda$. Let $[{\bf z}]$ denote
the equivalence class of ${\bf z}$. Let
$\P:\H^{n,1}-\{0\}\longrightarrow \H\P^n$ be the {\sl right projection} map given by $\P:{\bf z}\longmapsto z$, where $z=[{\bf z}]$.
The $n$ dimensional {quaternionic hyperbolic space} is defined to be
${\bf H}_{\H}^n=\P(V_-)$ with boundary $\partial {\bf H}_{\H}^n=\P(V_0)$.

\medskip In the model using $J_2$, there are two distinct points $0$ and $\infty$ on $\partial \h^n$.  For $z_1 \neq 0$,  the projection map $\P$ is given by 
$$\P(z_1, z_2, \ldots, z_{n+1})=(z_2z_1^{-1}, \ldots, z_{n+1}z_1^{-1}), $$
and  accordingly we choose boundary points 
\begin{equation}\label{infty}
\P(0, 1, \,\ldots,\,0,0)^t=0.
\end{equation}
\begin{equation}
\P(1, 0, \,\ldots,\,0,0)^t=\infty.
\end{equation}
In the model using $J_1$, we mark $\P(1, 0, \,\ldots,\,0,0)^t$ as the origin $0=(0,0, \ldots, 0)^t$ of the quaternionic hyperbolic ball.  The Bergmann metric on
$\h^n$ is given by the distance formula
$$
\cosh^2\frac{\rho(z,w)}{2}=
\frac{\langle{\bf z},\,{\bf w}\rangle \langle{\bf w},\,{\bf z}\rangle}
{\langle{\bf z},\,{\bf z}\rangle \langle{\bf w},\,{\bf w}\rangle},
\ \ \text{where}\ \ z,w \in \h^n, \ \
{\bf z}\in \P^{-1}(z),{\bf w}\in \P^{-1}(w). 
$$
The above forumula is independent of the choice of ${\bf z}$ and ${\bf w}$. 

Now consider the non-compact linear Lie group   
$$\Sp(n, 1) = \{A \in {\rm GL}(n + 1, \H) : A^*J_{i}A = J_{i}\}.$$
An element  $g\in{\rm Sp}(n,1)$ acts on
$\overline \h^n={\bf H}_\H^n\cup\partial \h^n $ as $
g(z)=\P g\P^{-1}(z)$. Thus the isometry group of $\h^n$ is given by 
${\rm PSp}(n,1)=\Sp(n,1)/\{I, -I\}.$

\subsection{Cao-Parker Inequality} 
Recall that the quaternionic cross ratio of four distinct points $z_1, z_2, z_3, z_4$ on $\partial \h^n$ is defined as:
$$[z_1, z_2, z_3, z_4]=\langle \z_3, \z_1 \rangle \langle \z_3, \z_2 \rangle^{-1} \langle \z_4, \z_2 \rangle \langle \z_4, \z_1 \rangle^{-1},$$
where $\z_i$ denote the lift to $\H^{n+1}$ of a point $z_i$ on $\partial \h^n$. 
We note the following lemma concerning cross ratios. 
\begin{lemma}\cite{cp1}
Let $0, \infty \in \partial \h^n$ stand for the $(0, 1, \dots , 0)^t$ and $(1, 0, \dots, 0)^t \in \H^{n, 1}$ under the projection map $\P$, respectively and let $h \in {\rm PSp}(n, 1)$ be given by (\ref{form}). Then

$$|[h(\infty), 0, \infty, h(0)]| = |bc|,$$
$$|[h(\infty), \infty, 0, h(0)]| = |ad|,$$
$$|[\infty, 0, h(\infty), h(0)]| = \frac{|bc|}{|ad|}.$$
\end{lemma}
Now, Cao and Parker's theorem may be stated as follows. 
\begin{theorem} {\rm (Cao and Parker)  \cite{cp1}}\label{cpt}
Let $g$ and $h$ be elements of $\Sp(n, 1)$ such that $g$ is loxodromic element with fixed points $u, v \in \partial \h^n$, and $M_g<1$. If $\langle g, h \rangle$ is non-elementary and discrete, then 
\begin{equation}\label{eql} |[h(u), u, v, h(v)]|^{\frac{1}{2}}|[h(u), v, u, h(v)]|^{\frac{1}{2}} \geq \frac{1 - M_g}{M_g^2}. \end{equation} 
\end{theorem}

\subsection{A Shimizu's Lemma in  $\Sp(n,1)$} \label{shi}
 We use the Hermitian form $J_2$ in this section. Up to conjugacy, we assume that an Heisenberg translation fixes the boundary point $0$, i.e. it is of the form
\begin{equation} \label{he} 
T_{s, \zeta}=\begin{pmatrix} 1 & 0 & 0 \\ s & 1 & {\zeta}^{\ast} \\ \zeta & 0 & I \end{pmatrix}, 
\end{equation}
where $Re(s)=\frac{1}{2} |\zeta|^2$.

 Let $A$ be an element in $\Sp(n,1)$. Then one can choose $A$ to be of the following form. 
\begin{equation}\label{form}
A=\begin{pmatrix} a & b & \gamma^{\ast} \\ c & d & \delta^{\ast} \\ \alpha & \beta & U \end{pmatrix},  \end{equation}
where $a, b, c, d$ are scalars, $\gamma, \delta, \alpha, \beta$ are column matrices and $U$ is an element in $M(n-1, \H)$. 
Then, it is easy to compute that 
$$A^{-1}=\begin{pmatrix} \bar d & \bar b & -\beta^{\ast} \\ \bar c & \bar a & -\alpha^{\ast} \\ -\delta & -\gamma & U^{\ast} \end{pmatrix}.$$

 The following theorem follows by mimicking the arguments of Hersonsky and Paulin in  \cite[Appendix]{hp}. However, Hersonsky and Paulin proved it over the complex numbers. To write it down over the quaternions, only slight variation is needed, and is straight-forward. We skip the details. 
\begin{theorem}\label{sht}
Suppose $T_{s, \zeta}$ be an Heisenberg translation in $\Sp(n,1)$ and $A$ be an element in $\Sp(n,1)$ of the form \eqnref{form}. Suppose $A$ does not fix $0$. Set
\begin{equation}t={\rm Sup}\{|b|, |\beta|, |\gamma|, |U-I| \}, \ M=|s|+2|\zeta|. \end{equation}
If  $M t +2|\zeta| < 1,$
then the group generated by $A$ and $T_{s, \zeta}$ is either non-discrete or fixes $0$. 
\end{theorem}
This is the simplest quaternionic version of the Shimizu's lemma for two generator subgroups of $\Sp(n,1)$ with a unipotent parabolic generator. More generalized versions of  the Shimizu's lemma in $\Sp(n,1)$ has been obtained by Kim and Parker in \cite[Theorem 4.8]{kp}, and  Cao and Parker \cite{cp2}.  Though the above version of the quaternionic Shimizu's lemma is simpler, it is weaker than the versions of Kim and Parker, and Cao and Parker. However, we find it easier to apply for our purpose. 

\subsection{Useful Results}\label{dstu}
 A subgroup $G$ of $\Sp(n, 1)$ is called elementary if it has a finite orbit in ${\bf H}_\H^n\cup\partial{\bf H}_\H^n$. If all of its
orbits are infinite then $G$ is non-elementary. In particular, $G$ is non-elementary if it contains two non-elliptic elements of infinite order with distinct fixed points.

\begin{theorem}\label{dstu1}\cite{chen}
 Let $G$ be a  Zariski dense subgroup of $\Sp(n,1)$. Then $G$ is either discrete or dense in $\Sp(n,1)$.
\end{theorem}
\section{Cao-Tan Inequality Revisited }\label{ctjo} 

\begin{theorem} \label{eljo}
Let $g$ and $h$ be elements of $\Sp(n, 1)$. Suppose that $g$ is a regular elliptic element with fixed point $q$, and $\delta(g)$ as in \eqnref{eqq1}.  If 
\begin{equation} \label{ct}
\cosh \frac{\rho(q, h(q))}{2}~ \delta(g) <1, 
\end{equation} 
then the group $\langle g, h \rangle$ generated by $g$ and $h$ is either elementary or not discrete.
\end{theorem}
The proof of the above theorem is a variation of the proof of \cite[Theorem 1.1]{ct1}. The initial computations are very similar, except that at a crucial stage we replace the Cao-Tan invariant by $\delta(g)$ and observe that it still works. We sketch the proof for completeness. We follow similar notations as in \cite{ct1}. We shall use the ball model, i.e. Hermitian from $J_1$ is being used is what follows:

\begin{proof} Using conjugation, we may assume that $g$ is of the following form (\ref{elp}) having fixed point $q=(0,\ldots, 0)^t \in \h^n$ and

$$ h = (a_{i,j})_{i, j = 1, \dots , n+1} = \begin{pmatrix}    a_{1, 1} & \beta \\ \alpha & A \end{pmatrix}.$$ 
For  $L=diag(\lambda_2, \ldots, \lambda_{n+1})$, write $g$ as:
$$g=\begin{pmatrix} \lambda_1 & 0 \\ 0 & L \end{pmatrix}.$$
Then 
$$ \cosh \frac{\rho(q, h(q))}{2} = |a_{1, 1}|, \; \delta(g) = \max\{ \ |\lambda_1-1|+|\lambda_i-1| \ : \ i=2, \ldots, n+1\}.$$
The inequality \eqnref{ct}  becomes, 
\begin{equation} \label{equa0}
 |a_{1, 1}|\delta(g) < 1. 
\end{equation}

Let $h_0 = h$ and $h_{k+1} = h_kgh_k^{-1}.$ We write
$$h_k = (a_{i,j}^{(k)})_{i, j = 1, \dots , n+1}= \begin{pmatrix}    a_{1, 1}^{(k)} & \beta^{(k)}\\\alpha^{(k)} & A^{(k)}\end{pmatrix}.$$
If for some $k$, $\beta^{(k)} =0$, as in the proof of \cite[Theorem 1.1]{ct1}, it follows that $\langle g, h \rangle$ is elementary. So, assume $\beta^{(k)} \neq 0$ and the group $\langle g, h \rangle$ is non-elementary. Then following exactly similar computations as in the proof of \cite[Theorem 1.1]{ct1}, one can see that: 
\begin{equation} \label{equ1}
|a_{1, 1}^{(k+1)}|^2  \leq  |a_{1, 1}^{(k)}|^4 + |\beta^{(k)}|^4 - \sum_{i=2}^{n+1}|a_{1, 1}^{(k)}|^2|a_{1, i}^{(k)}|^2(2-|u_1 - u_{i}|^2),
\end{equation}
where 
 $$u_i = \overline{a_{1, i}^{(k)}}^{-1}\lambda_i \overline{a_{1, i}^{(k)}}, \; \; i = 2, \dots n+1.$$

Noting that $ |a_{1, 1}^{(k)}|^2 - |\beta^{(k)}|^2 =1$, by (\ref{equ1}) we have
\begin{eqnarray} |a_{1, 1}^{(k+1)}|^2 -1 &  \leq &  |a_{1, 1}^{(k)}|^2 \sum_{i=2}^{n+1} |a_{1, i}^{(k)}|^2|u_1 - u_{i}|^2\\ 
 &  \leq &  |a_{1, 1}^{(k)}|^2 \sum_{i=2}^{n+1} |a_{1, i}^{(k)}|^2| \bigg(|u_1-1|^2 + |u_{i}-1|^2\bigg)\\
 & \leq &   |a_{1, 1}^{(k)}|^2 \sum_{i=2}^{n+1} |a_{1, i}^{(k)}|^2~ \bigg(|u_1-1| + |u_{i}-1|\bigg)^2. 
\end{eqnarray} 
Therefore
\begin{equation} \label{equa2}
|a_{1, 1}^{(k+1)}|^2 -1 \leq (|a_{1, 1}^{(k)}|^2 -1)~|a_{1, 1}^{(k)}|^2  \delta^2(g).
\end{equation}
Now, it follows by induction that 
$$|a_{1, 1}^{(k+1)}| < |a_{1, 1}^{(k)}|, $$
and
\begin{equation} \label{equa3}
|a_{1, 1}^{(k+1)}|^2 -1 < (|a_{1, 1}|^2 - 1) (|a_{1, 1}|^2 \delta^2(g))^{k+1}.
\end{equation}

Since $|a_{1, 1}|\delta(g)<1$,  $|a_{1, 1}^{(k)}| \to 1$.  Now, as in the last part of the proof of \cite[Theorem 1.1]{ct1}, 
$$ \beta^{(k)} \to 0, \; \alpha^{(k)} \mapsto 0, ~~A^{(k)}(A^{(k)})^{\ast} \to I_n. $$

By passing to its subsequence, we may assume that 

$$A^{(k_t)}\to A_{\infty}, \; \; a_{1, 1}^{(k_t)} \to a_{\infty}.$$
Thus $h_{k+1}$ converges to 
$$h_{\infty} = \begin{pmatrix} a_{\infty} & 0 \\ 0 & A_{\infty} \end{pmatrix} \in \Sp(n, 1), $$
which implies that $\langle g, h \rangle$ is not discrete. This completes the proof.
\end{proof} 

Using embedding of $\SL(2, \C)$ in $\Sp(1,1)$ and then applying similar arguments as in the proof of \cite[Theorem 1.2]{ct1}, we have the following corollary that may be thought of a generalized version of the classical J\o{}rgensen inequality in $\SL(2, \C)$ for two generator subgroups with an elliptic generator. 
\begin{cor}\label{jcor1} 
Let $g$ and $h$ are elements in $\SL(2, \C)$. Let 
$$g=\begin{pmatrix} e^{i \theta} & 0 \\ 0 & e^{-i \theta} \end{pmatrix},~ \theta \in [0, \pi], ~~h=\begin{pmatrix} a & b \\ c & d \end{pmatrix}.$$
Let $||h||^2=|a|^2 + |b|^2 + |c|^2 +|d|^2$. If $\langle g, h \rangle$ is non-elementary and discrete, then
\begin{equation} \label{eq111} 4 \sin^2 \frac{\theta}{2} \bigg(||h||^2 +2 \bigg)\geq 1.\end{equation} 
\end{cor} 

\begin{proof}  Let $\hat g$ be the image of $g$ in $\Sp(1,1)$. Then we have in the above notation, using similar calculations as in \cite[Section 4]{ct1}, 
$$\delta(\hat g) = 4 \sin\frac{\theta}{2}, $$
and $\cosh^2(\frac{\rho(0, \hat h(0)}{2})=||h||^2$. This gives the proof. 
\end{proof}

\subsection{Comparison of the conjugacy invariants}
Let $g$ be ellliptic, up to conjugacy, in $\Sp(n,1)$, 
$$g=diag(\lambda_1, \ldots, \lambda_{n+1}), $$
where $|\lambda_i|=1$ for all $i$. Cao and Tan used the following conjugacy invariant instead of $\delta(g)$: 
$$\delta_{ct}(g)=\max\{|\lambda_i-\lambda_1|^2 \ : \ i=2, \ldots, n+1\}.$$

We have  
$$\delta(g)=\max\{ \ |\lambda_1-1|+|\lambda_i-1| \ : \ i=2, \ldots, n+1\}. $$
For all $j$, let $\lambda_j=e^{i \theta_j}$, $\theta_j \in [0, \pi]$. Note that 
$$
|e^{i\theta} -1| + |e^{i\phi} - 1| = 2\big (|\sin\frac{\theta}{2}| + |\sin\frac{\phi}{2}| \big), $$
this implies 
\begin{eqnarray*} 
\delta(g)&=& 2 \max\{|\sin\frac{\theta_1}{2}| + |\sin\frac{\theta_{j+1}}{2}| : j=1, \dots , n\}\\
&=& \max\{2 (\sin\frac{\theta_1}{2} + \sin\frac{\theta_{j+1}}{2}) : j=1, \dots , n\} \\
&=& \max\{ 4\sin\frac{\theta_1 + \theta_{j+1}}{4} \cos\frac{\theta_1 - \theta_{j+1}}{4} :  j=1, \dots , n\}.\\
\end{eqnarray*}
On the other hand, the expression for Cao-Tan invariant in \cite{ct1} is 
$$\delta_{ct}(g) = \max\left\lbrace 4\sin^2\frac{\theta_1 \pm \theta_{j+1}}{2} : j=1, \dots , n\right\rbrace.$$
 
 Recall that by \cite[Corollary 1.2]{ct1}, under the hypothesis of the above corollary, 
\begin{equation} \label{eq12} 4 \sin^2 \theta \bigg(||h||^2 +2 \bigg)\geq 1.\end{equation} 
Comparing the two sine terms in the LHS of the inequalities \eqnref{eq111} and \eqnref{eq12}, we see that 
$$\sin^2(\theta/2) \leq  \sin^2 \theta, \hbox{ for } \theta \in [0, 2 \pi/3],$$
and our inequality \eqnref{eq111} is stronger than the inequality \eqnref{eq12} of Cao and Tan. But when $\theta \in (2 \pi/3, \pi]$, then $\sin^2(\theta/2) > \sin^2 \theta$, and consequently the inequality of Cao and Tan is better in this subinterval. So, except the last one-third of the interval $[0, \pi]$, our version of the J\o{}rgensen inequality in $\SL(2, \C)$  is better. 

\section{Proof of \thmref{thd} } \label{pfthd}

\begin{proof}
 Given $g$, let $F_g$ denote the subgroup of $\Sp(n,1)$ that stabilizes the set of  fixed points of $g$. The subgroup $F_g$ is closed in $\Sp(n,1)$. 

If possible suppose $G$ is not discrete. Then $G$ is dense in $\Sp(n,1)$, by \thmref{dstu1}. 
Since the set of loxodromic elements $\mathcal L$ forms an open subset of $\Sp(n,1)$,  $\mathcal L \setminus F_g$ is also an open subset in $\Sp(n,1)$. 

\medskip 
(1) Let $g$ be a regular elliptic. We shall use the ball model. Up to conjugacy, we may assume that $q=0$ is a fixed point of $g$, and it is of the form \eqnref{elc}. Since, $G$ is dense in $\Sp(n,1)$, we can get a sequence of loxodromic elements $\{h_m\}$ in $\mathcal L \cap G$ such that $h_m \to I$. For each $m$, the element $h_m gh_m^{-1}$ is also a regular elliptic with fixed point $h_m(q)$. Let 
\begin{equation} \label{eee1} h_m g h_m^{-1} = (a_{i,j}^{(m)}) = {\begin{pmatrix}    a_{1, 1}^{(m)} & \beta^{(m)}\\\alpha^{(m)} & A^{(m)} \end{pmatrix}}.\end{equation} 
Then $h_m g h_m^{-1} \to g$. In particular, $a_{1,1}^{m} \to \lambda_1$, where $|\lambda_1|=1$.  Since $q=0$ is a fixed point of $g$, the left hand side of (\ref{ct}) becomes $|a_{1, 1}^{(m)}| \delta(g)$. The group $\langle g, h_m g h_m^{-1}\rangle$ is clearly discrete. 

 If possible, suppose $\langle g, h_m g h_m^{-1} \rangle$  is elementary. Since loxodromic elements have no fixed point on $\h^n$, $h_m(0) \neq 0$. Thus, $g$ and $h_mg h_m^{-1}$ do not have a common fixed point. Then it must keep two boundary points $p_1$, $p_2$ invariant, and hence will keep invariant the quaternionic line $l$ passing through $p_1$ and $p_2$. Then $g|_l$ acts as a regular elliptic element of $ {\rm Isom}(l) \approx \Sp(1,1)$. Hence, $q$ must belong to $l$, otherwise $g$ would have at least two fixed points contradicting regularity of $g$. Now, note that $g^2|_l$ is also an elliptic element that fixes  $p_1$, $p_2$ and $q$. Then it can be seen that with respect to a chosen basis ${\bf p}_1$ and ${\bf p}_2$, $g^2|_l$ must be of the form $g^2|_l=diag (\lambda, \lambda)$, where $|\lambda|=1$. This implies $g$ has an eigenvalue class represented by $\lambda^{1/2}$ of multiplicity at least $2$. This again contradicts the regularity of $g$. 

So, the group  $\langle g, h_m g h_m^{-1}  \rangle$ must be non-elementary. By \thmref{ct}, 
$$|a_{1, 1}^{(m)}| ~\delta(g) \geq 1. $$
But $|a_{1, 1}^{(m)}| \to 1$ and $\delta(g) <1$.  So, the above is a contradiction. This proves part (1). 

\medskip We shall use the Siegel domain model for proving the other assertions. 

\medskip (2) Let $g$ be loxodromic. Up to conjugacy, let $0$ and $\infty$ are the fixed points of $g$, and so $g$ is of the form \eqnref{elp}.  Since $G$ is dense in $\Sp(n, 1)$,  there exist a sequence $\{h_n\}$ of loxodromic elements in $(\mathcal L \setminus F_g)\cap G$ such that $h_n \rightarrow g$. 
Let
\begin{equation} \label{eqq1} h_n g h_n^{-1} =
\begin{pmatrix}
     a_n & b_n & \gamma_n^*\\
c_n  & d_n & \delta_n^*\\
    \alpha_n &  \beta_n & U_n\\
     \end{pmatrix}.\end{equation} 

 Since, $h_n \in \mathcal L \setminus F_g$,  $g$ and $h_n$ can not have a common fixed point, and neither can have a two point invariant subset. So,  and $\langle g, h_n g h_n^{-1} \rangle$ is non-elementary for each $n$. By \thmref{cpt}, 
    $$|a_nd_n|^{\frac{1}{2}} |b_nc_n|^{\frac{1}{2}} \geq \frac{1 - M_f}{M_f^2}.$$
    But $b_nc_n \rightarrow 0$ as $n \to \infty $, hence 
    $$\frac{1 - M_f}{M_f^2} \leq 0, $$
    which is a contradiction.

\medskip (3) Let $g$ be a Heisenberg translation. Up to conjugacy, let $0$ be the fixed point of $g$ and it is of the form \eqnref{he}. As $g \in \bar{G}$, there exist a sequence of loxodromic elements $\{h_n\} \in (\mathcal L \setminus F_g)\cap G$ such that 
$$ h_n \rightarrow g.$$
 Let $h_n g h_n^{-1}$ be of the form \eqnref{eqq1}. Since  $h_n g h_n^{-1} \to g$, it follows that $t_n \to 0$. 

Since $g$ and $h_n g h_n^{-1}$ have no fixed points in common, $\langle g, h_n g h_n^{-1} \rangle$ is discrete and non-elementary, hence by \thmref{sht}, 
$$Mt_n + 2 |\zeta| > 1.$$
But $t_n \to 0$ as $n \to \infty$. Thus for large $n$, $|\zeta|\geq \frac{1}{2}$. 
This is a contradiction as $|\zeta| < \frac{1}{2}$ is given.  

This proves the theorem.
\end{proof}
\subsection{Proof of \corref{cor1}}
 Note that the set of regular elliptic elements in $\Sp(n,1)$ forms an open subset $\mathcal E$. 

(1) Let $g$ be a regular elliptic. We shall use the ball model. Up to conjugacy, we may assume $g$ is of the form \eqnref{elc}, and thus $g(0)=0$. Since, $G$ is dense in $\Sp(n,1)$, there exists a sequence of regular elliptic elements $\{h_m\}$ in $(\mathcal E \setminus F_g) \cap G$ such that $h_m \to I$. For each $m$, the element $h_m gh_m^{-1}$ is also a regular elliptic with fixed point $h_m(0)$. Let $h_m g h_m^{-1}$ is of the form \eqnref{eee1}. 
The group $\langle g, h_m g h_m^{-1}\rangle$ is clearly discrete. We claim that it is also not elementary. For otherwise, $g$ and $h_m g h_m^{-1}$ must have a common fixed point that will be different from $0$ and $h_m(0)$, which will contradict the regularity of the isometries. Now, by \thmref{ct}, 
$|a_{1, 1}^{(m)}| ~\delta(g) \geq 1$. 
Since $|a_{1, 1}^{(m)}| \to 1$ and $\delta(g) <1$, this is a contradiction. This proves part (1). 

\medskip Using similar arguments as in the proof of  \thmref{thd}, (2) and (3) follow.

\end{document}